\documentclass[11pt]{amsart}
\usepackage{amsfonts,amssymb,amsmath,amsthm}
\usepackage[english]{babel}
\usepackage{mathrsfs}
\usepackage{geometry}
\usepackage{enumerate}
\usepackage{graphicx,color}

\theoremstyle{plain}

\theoremstyle{plain}
\newtheorem{theorem}{Theorem} [section]

\newtheorem{lemma}[theorem]{Lemma}
\newtheorem{proposition}[theorem]{Proposition}

\theoremstyle{definition}

\newtheorem{remark}[theorem]{Remark}
\numberwithin{theorem}{section}
\numberwithin{equation}{section}
\numberwithin{figure}{section}

\def\mean#1{\mathchoice
         {\mathop{\kern 0.2em\vrule width 0.6em height 0.69678ex depth -0.58065ex
                 \kern -0.8em \intop}\nolimits_{\kern -0.4em#1}}%
         {\mathop{\kern 0.1em\vrule width 0.5em height 0.69678ex depth -0.60387ex
                 \kern -0.6em \intop}\nolimits_{#1}}%
         {\mathop{\kern 0.1em\vrule width 0.5em height 0.69678ex
             depth -0.60387ex
                 \kern -0.6em \intop}\nolimits_{#1}}%
         {\mathop{\kern 0.1em\vrule width 0.5em height 0.69678ex depth -0.60387ex
                 \kern -0.6em \intop}\nolimits_{#1}}}

\def\R{\mathbb R}

\def\a{\alpha}
\def\aalpha{{\mbox{\boldmath$\alpha$}}}

\def\e{\varepsilon}

\newcommand{\p}{\partial}

\def\U{\mathcal O}

\def\A{{\bf{\mathcal A}}}
\def\B{{\bf{\mathcal B}}}
\def\C{{\bf{\mathcal C}}}
\def\normm{|\!|\!|}

\DeclareMathOperator*{\osc}{osc}

\title{Sobolev regularity for Monge-Amp\`ere type equations}

\author[G. De Philippis]{Guido De Philippis}
\address{Scuola Normale Superiore,
p.za dei Cavalieri 7, I-56126 Pisa, Italy}
\email{guido.dephilippis@sns.it}

\author[A. Figalli]{Alessio Figalli}
\address{The University of Texas at Austin,
Mathematics Dept. RLM 8.100,
2515 Speedway Stop C1200,
Austin, Texas 78712-1202, USA}
\email{figalli@math.utexas.edu}


\begin{document}
\begin{abstract}
In this note we prove that,
if the cost function satisfies some necessary structural conditions
and the densities are bounded away from zero and infinity,
then strictly $c$-convex potentials
arising in optimal transportation 
belong to $W^{2,1+\kappa}_{\rm loc}$ for some $\kappa>0$.
This generalizes some recents results \cite{DepFi,DFS,S} concerning 
the regularity of strictly convex Alexandrov solutions of the Monge-Amp\`ere equation with right hand side bounded away from zero and infinity.
\end{abstract}

\maketitle

\section{Introduction}

Let $\Omega\subset \R^n$ be a bounded open set. We want to investigate the regularity of solutions to Monge-Amp\`ere type equations of the form
\begin{equation}
\label{eq:MA1}
\det\big(D^2u-\A(x,Du)\big)=f \qquad \text{in $\Omega$},
\end{equation}
where $f \geq 0$ and $\A(x,p)$ is a $n\times n$ symmetric matrix.

This class of equations naturally arises in optimal transportation, and in reflector and refractor shape design problems.
In these applications, the matrix $\A$ and the right and side $f$ are given by
$$
\A(x,D u(x))=-D_{xx} c(x,T_u(x)), \qquad f(x)=\left|\det\bigl(D_{xy}c(x,T_u(x))\bigr)\right| \frac{\rho_0(x)}{\rho_1(T_u(x))},
$$
where $c(x,y)$ represents the cost function, $\rho_0$ and $\rho_1$ are
probability densities, and $T_u$ is the optimal transport map
sending $\rho_0$ onto $\rho_1$.
Under a twist assumption on the cost (see {\bf (C2)} below), the map $T_u$ is uniquely determined
through the relation
$$
-D_xc(x,T_u(x))=Du(x).
$$
Moreover, when $\A\equiv 0$ the above equation reduces to the classical Monge-Amp\`ere equation.

The regularity for the above class of equations has received a lot of attention in the last years \cite{FKM,FL,Liu,LTWC2a,Loe,MTW,TW1,TW2}.
In particular, under some necessary structural conditions on $\A$ (see {\bf (C1)} below),
one can show that if \(f\) is smooth then \(u\) is smooth as well \cite{LTWC2a,MTW,TW1,TW2}. In
addition, it is proved in \cite{FKM} that solutions are locally $C^{1,\alpha}$
when $f$ is merely  bounded away from zero and infinity (see also \cite{FL,Liu}).\\

Recently, the authors introduced new techniques to address the Sobolev regularity of $u$
when $\A\equiv 0$:
more precisely, under the assumption that $f$ is bounded away from zero and infinity,
it is proved in \cite{DepFi} that $D^2 u \in L\log L_{\rm loc}(\Omega)$,
and with a variant of the same techniques this result has been improved in \cite{DFS} to 
$u \in W^{2,1+\kappa}_{\rm loc}(\Omega)$ for some $\kappa>0$ (see also \cite{S}). 
Let us mention that these results played a crucial role in \cite{ACDF1,ACDF2} to show the
 existence of distributional solutions to the semi-geostrophic system.

The aim of this paper is to extend the $W^{2,1+\kappa}_{\rm loc}$ regularity
to the general class of Monge-Amp\`ere equations
in \eqref{eq:MA1}. Apart from its own interest, it seems likely that this result
could have applications in the study of generalized semi-geostrophic system on Riemannian manifolds \cite{CDRS}, in particular on the sphere \cite{Loe2} and its perturbations \cite{DG1,DG2,FR,FRVsn}.

In order to describe our result, we need to introduce some more notation and the main assumptions on the cost functions.\\

Let $X\subset \R^n$ be an open set,
and $u:X \to \R$ be a $c$-convex function, i.e., $u$ can be written as 
\begin{equation}
\label{eq:cconv}
u(x)=\max_{y \in \overline Y} \{-c(x,y) +\lambda_y\}
\end{equation}
for some open set $Y\subset \R^n$, and $\lambda_y \in \R$ for all $y \in \overline Y$.
We are going to assume that $u$ is an Alexandrov solution of \eqref{eq:MA1}
inside some open set $\Omega \subset X$, i.e.,
$$
\left|\p^cu(E) \right|= \int_E f\qquad \text{for all $E\subset \Omega$ Borel},
$$
where 
$$
\p^cu(E) :=\bigcup_{x \in E}\p^cu(x), \qquad 
\p^cu(x):=\{y \in \overline Y \,:\,u(x)=-c(x,y)+\lambda_y\},
$$
and $|F|$ denotes the Lebesgue measure of a set $F$.
It is well-known that, in order to prove some regularity results,
\eqref{eq:MA1} needs to be coupled with some boundary
conditions:
for instance, when $\A\equiv 0$ and $f\equiv 1$,
solutions are smooth whenever they are strictly convex,
and to obtain strict convexity some suitable boundary conditions are needed \cite{CA1,CA3}. 

For the general case in \eqref{eq:MA1},
let $u$ be a $c$-convex function associated to an optimal transport problem,
and for any $y \in \overline Y$ define the contact set
$$\Lambda_y:=\{x \in X \,:\,u(x)=-c(x,y)+\lambda_y\}.$$ 
Under some structural assumptions on the cost functions (which we shall describe below)
and some convexity hypotheses on the supports of the source and target measure,
it has been proved in \cite{FKM} that $u$ is an Alexandrov solution of \eqref{eq:MA1}
inside $X$,
and it is strictly $c$-convex (i.e., for any $y \in \p^cu(X)$ the contact set
$\Lambda_y$ reduces to one point)
 provided $f$ is bounded away from zero and infinity.

Here, since we want to investigate the interior regularity of $u$,
instead of assuming that $u$ comes from an optimal transportation problem 
where the supports of the source and target measure enjoy some global ``$c$-convexity'' property,
we work assuming directly that $u$ is a strictly $c$-convex Alexandrov solution
near some point $\bar x \in X$,
and we prove regularity of $u$ in a neighborhood of $\bar x$.
This has the advantage of making our result more general and flexible for possible future applications.

Hence, we assume that there exist $(\bar x, \bar y)\in X\times Y$
such that $\Lambda_{\bar y}=\{\bar x\}$, we consider a neighborhood $\Omega$
of $\bar x$ given by
\begin{equation}
\label{eq:omega}
\Omega:=\{x \in X\,:\,u(z)<-c(x,\bar y)+\lambda_y+\delta\},
\end{equation}
where $\delta>0$ is a small constant chosen so that $\Omega \subset \subset X$
and $\partial^c u(\Omega)\subset \subset Y$
(such a constant $\delta$ exists because $\Lambda_{\bar y}:=\{\bar x\}$). 
Also, we assume that $u$ is an  Alexandrov solution of
\begin{equation}
\label{eq:MA}
\left\{
\begin{array}{ll}
\det\big(D^2u-\A(x,Du)\big)=f & \text{in $\Omega$},\\
u=-c(\cdot,\bar y)+{\rm const} & \text{on $\partial \Omega$}.
\end{array}
\right.
\end{equation}

Before stating our result, let us introduce the main conditions on the cost function:
let $\Omega$ be as above, and let $\U\subset\subset Y$ be a open neighborhood of 
$\partial^c u(\Omega)$. We define
\begin{equation}\label{cnorm}
\normm c\normm :=\|c\|_{C^3(\overline \Omega \times \overline \U)}+\|D_{xxyy}c\|_{L^\infty(\overline \Omega \times \overline \U)}+\left\|\log |\det D_{xy}c|\right\|_{L^\infty(\overline \Omega \times \overline \U)},
\end{equation}
and assume that the following hold: 
\medskip
\begin{enumerate}
\item [{\bf (C0)}]  $\normm c\normm <\infty$. 
\item[{\bf (C1)}] For every $x\in \Omega$ and $p:=-D_xc(x,y)$ with $y \in \U$, it holds
\begin{equation}\label{MTWc}
D_{p_kp_\ell}\A_{ij}(x,p)\xi_i\xi_j\eta_k \eta_\ell\ge 0,\qquad \forall \, \xi,\eta \in \R^n,\, \xi \cdot \eta=0,
\end{equation}
where $\A$ is defined through $c$ by $\A_{ij}(x,p):=-D_{x_ix_j}c(x,y)$, 
and we use the summation convention over repeated indices.
\end{enumerate}
\smallskip

Let us point out that, up to reduce the size of
$\Omega$ and $\U$ (this is possible because $\Omega \to \{\bar x\}$
and $\partial^cu(\Omega) \to \{\bar y\}$ as $\delta\to 0$),
as a consequence of {\bf (C0)} (more precisely, from the fact that
$\det D_{xy}c(\bar x,\bar y)\neq 0$
and by the implicit function theorem) we can assume that the following holds:
\medskip
\begin{enumerate}
\item[{\bf (C2)}] For every $(x,y)\in \Omega\times \U$, the maps
$x\in \Omega \mapsto - D_yc(x,y)$ and $y\in \U \mapsto - D_xc(x,y)$ are
diffeomorphisms on their respective ranges. 
\end{enumerate}

\smallskip

We also notice that, because of the boundary condition $u=-c(\cdot,\bar y)+{\rm const}$ on $\partial \Omega$, if $f$ is bounded away from zero and infinity inside $\Omega$, 
then any  $c$-convex Alexandrox solution of \eqref{eq:MA} is strictly $c$-convex inside $\Omega$
(this is an immediate consequence of \cite[Remark 7.2]{FKM}).
Here is our result:

\begin{theorem}\label{w21eps}Let $u:\Omega \to \R$ be a $c$-convex
Alexandrov solution of \eqref{eq:MA}.
Assume that $c$ satisfies conditions {\bf (C0)}-{\bf (C2)}, and that $0<\lambda\le f\le 1/\lambda$. Then $u\in W_{\rm loc}^{2,1+\kappa}(\Omega)$ for some $\kappa>0$.
\end{theorem}

Theorem \ref{w21eps} 
generalizes the corresponding result for the classical Monge-Amp\`ere
equation to the wider class of equations considered here.
With respect to the arguments in \cite{DepFi,DFS}, 
 additional complications arise from the fact that, in contrast with the classical Monge-Amp\`ere equation, in general
\eqref{eq:MA1} is not affinely invariant.\\

\textit{Acknowledgements:} AF is partially supported by NSF Grant DMS-0969962.
Both authors acknowledge the support of the ERC ADG Grant GeMeThNES.
The first author thanks the hospitality of the Mathematics Department at the University
of Texas at Austin, where part of this work has been done.

\section{Notation and preliminary results}
Through all the paper, we call \emph{universal} any constant which depends only on the data, i.e., on $n$, $\Omega$, $\U$, $\lambda$, and $\normm
c\normm$. We use $C$ to denote a universal  constant larger than $1$ whose value may change from line to line, and we use the notation $a\approx b$ to indicate that the ratio $a/b$ is bounded from above and below by positive universal constants.\\

An immediate consequence of the definition of $c$-convexity \eqref{eq:cconv} is that,
for any $x_0\in X$, there exists $y_0\in \overline Y$ such that
\[
u(x)\ge - c(x,y_0)+u(x_0)+c(x_0,y_0) \quad \forall \,x\in X,
\]
and in this case $y_0 \in \p^cu(x_0)$.
If in addition $u \in C^2$, then it is easily seen that $Du(x_0)=-D_xc(x_0,y_0)$
and  $D^2u(x_0)\ge -D_{xx} c(x_0, y_0)=\A(x_0,Du(x_0))$, where $\A$ is defined in {\bf (C1)} above.
In particular equation \eqref{eq:MA} is degenerate elliptic when restricted to $c$-convex function.

It has been discovered independently in \cite{FKM} and \cite{Liu} that,
because of {\bf (C1)},
for any $x_0 \in \Omega$
and $y_0 \in \p^cu(x_0)$,
through the change of variables $x\mapsto q(x):=-D_y c (x, y_0)$ the function
\begin{equation}\label{baru}
\bar u(q):=u(x(q))+c(x(q),y_0)-u(x_0)-c(x_0,y_0)
\end{equation}
has convex level sets inside $\Omega$ (here and in the sequel $x(q)$ denotes the inverse of $q(x)$,
which is well defined because of {\bf (C2)}).
Moreover $\bar u$ is $\bar c$-convex, where 
\begin{equation}\label{cctilda}
\bar c(q,y):=c(x(q),y)-c(x(q),y_0),
\end{equation}
see \cite[Theorem 4.3]{FKM}.

Since $u$ solves \eqref{eq:MA} one can check by a
direct computation that $\bar u$ solves
\begin{equation}\label{MAbaru}
\det\big(D^2\bar u-\B(q,D\bar u)\big)=g,
\end{equation}
with 
\begin{equation}\label{eq:bij}
\B_{ij}(q,p)=-D_{q_iq_j} \bar c(q,T_{\bar u}(q)) \quad \text{and}
\quad g(q)=f(x(q))\left[\det D_{xy}c(x(q),T_{\bar u}(q))\right]^{-2},
\end{equation}
where $T_{\bar u}$
is the map uniquely identified by the relation
$D\bar u(q)=-D_qc(q,T_{\bar u}(q))$.
Moreover it holds
\begin{equation}\label{0inunpunto}
\B_{ij}(\cdot ,0)\equiv 0,\qquad D_{p}\B_{ij}(\cdot ,0)\equiv 0,
\end{equation}
so using Taylor's formula we can write
\begin{equation}\label{biist}
\B_{ij}(q,D\bar u)=\B_{ij,k\ell}(q, D\bar u)\partial_k \bar u\partial_l \bar u 
\end{equation}
where
\begin{equation}\label{bijst}
\B_{ij,k\ell}(q, D\bar u(q)):=\int_0^1 D_{p_kp_\ell}\B_{ij}(q,\tau D\bar u(q))\, d\tau. 
\end{equation}
In addition, since condition {\bf (C1)} is tensorial \cite{MTW,Loe,KM} and $\normm c\normm$ involves only mixed fourth derivative, it is easily seen that $\normm \bar c\normm \approx \normm  c\normm$ and $\B$ satisfies the same assumptions as $\A$.
In particular {\bf (C1)} and \eqref{bijst} imply that
\begin{equation}\label{MTWB}
\B_{ij,k\ell}\xi_i\xi_j \eta_k\eta_\ell\ge 0\quad \forall \, \xi\cdot \eta=0.
\end{equation}
\medskip

Given a $C^1$ $c$-convex function as above, for any $x_0\in \Omega$, $y_0=T_u(x_0)$,
and $h\in \R^+$, we define the \emph{section} centered at $x_0$ of height $h$ as 
\[
S^u_h(x_0):=\{x \in \Omega\,:\,\ u(x)\le -c(x,y_0)+u(x_0)+c(x_0,y_0)+h\}.
\]
Assuming that $S_h^u(x_0)\subset \Omega$,
through the change of variables $x \mapsto q(x):=-D_y c(x,y_0)$ this section is transformed into the \emph{convex} set
\[
Q^{\bar u}_h(q_0):=-D_y c(S^u_h(x_0),y_0)=\{q:\ {\bar u}(q)\le h\}.
\]
When no confusion arises, we will often abbreviate $S_h(x_0)$ and $Q_h(q_0)$ for $S^u_h(x_0)$ and $Q^{\bar u}_h(q_0)$.

We also recall \cite{john} that, given an open bounded convex set $Q$, there exists an ellipsoid $E$ such that
\begin{equation}
\label{eq:john}
E\subset Q\subset n E,
\end{equation}
where the dilation is done with respect to the center of $E$.
We refer to it as the John ellipsoid of $Q$,
and we say that $Q$ is normalized if $E=B(0,1)$.
An immediate consequence of \eqref{eq:john} is that any open bounded convex set  $Q$
admits  an affine  transformation $L$ such that $L(Q)$ is normalized.
Hence, given $u$ and $S_h^u(x_0)$ as above, we can consider $\bar u$,
his section $Q_h^{\bar u} (q_0)$, and the normalizing affine transformation $L$.
Then we define $\bar  w : L(Q_h) \to \R $ as 
\begin{equation}\label{barbarw}
 \bar w ( q'):=(\det L)^{2/n} \bar u(q), \qquad \,q':=Lq.
\end{equation}
It is easy to check that $\bar w$ solves
\begin{equation}\label{MAbarw}
\det\big(D^2 \bar w(q')-\C( q', D\bar w(q'))\big)=g(L^{-1}q'),
\end{equation}
where
\begin{equation*}
\C( q', D\bar w(q')):=(\det L)^{2/n}(L^*)^{-1} \B\left(L^{-1}q',(\det L)^{-2/n} L^*D \bar w(q')\right) L^{-1} 
\end{equation*}
Up to an isometry, we can assume that 
\[
E=\biggl\{q :\ \sum_{i=1}^n \frac{q_i^2}{r_i^2} \le 1\biggr\},
\]
with $r_1\le \ldots\le r_n$. Then  $L^{-1} ={\rm diag}(r_1, \dots,r_n)$, and
\begin{equation}\label{Cijst}
\C_{ij}( q',D\bar w(q'))=\C_{ij,k\ell}( q', D\bar w(q'))\partial_k \bar w\partial_\ell \bar w 
\end{equation}
with
\begin{equation}\label{CB}
\C_{ij,k\ell}(q',D\bar w(q'))=(r_1\dots r_n)^{2/n} \frac{r_ir_j}{r_k r_\ell }\B_{ij,k\ell}(q, D\bar u(q)),
\end{equation}
see \eqref{bijst}. Moreover, by \eqref{MTWB}
(or again because of the tensorial nature of condition {\bf (C1)})
\begin{equation}\label{MTWC}
\C_{ij,k\ell}\xi_i\xi_j \eta_k\eta_\ell\ge 0\quad \forall \, \xi\cdot \eta=0.
\end{equation}
\medskip

Still with the same notation as above,
we also define the \emph{normalized size} of a section $S_h(x_0)$ as
\begin{equation}\label{normalizedsize}
\aalpha(S_h(x_0))=\aalpha(Q_h(q_0)):=\frac{|L|^2}{(\det L)^{2/n}}.
\end{equation}
Notice that, even if $L$ may not be unique, $\aalpha$ is well defined up to universal constants. In case $u$ is $C^2$ in a neighborhood of $x_0$, by a simple Taylor expansion 
of $\bar u$ around $q_0$ it is easy to see that there exists \(h(x_0)>0\)
small such that
\begin{equation}\label{hessubar}
\aalpha(S_h(x_0))=\aalpha(Q_h(q_0))\approx |D^2 \bar u (q_0)|\qquad \forall\, h\le h(x_0),
\end{equation}
where  \(q_0:=q(x_0)\).
Since $u$ and $\bar u$ are related by a diffeomorphism, the following lemma holds:

\begin{lemma}\label{eccen}
Let $\Omega'\subset \Omega$, and $u\in C^2(\Omega')$ be a strictly $c$-convex function
such that $\|Du\|_{L^\infty(\Omega')}$ is universally bounded. Then there exists a universal constants $M_1$ such that the following holds: For every $x_0 \in \Omega$ there exists a height $\bar h(x_0)>0$ such that if $|D^2 u(x_0)|\ge M_1$, then 
\begin{equation}\label{eq:ecc}
|D^2 u(x_0)|\approx \aalpha (S_h(x_0)) \qquad \forall\, h\le \bar h(x_0).
\end{equation}
\end{lemma}

\begin{proof}
Differentiating twice the relation \eqref{baru} we obtain
\[
D_{qq} \bar u=D_{q} x D_{xx} u+D_{qq} x D_xu+D_q xD_{xx} c +D_{qq}x D_x c,
\]
which implies that
\begin{equation}
\label{eq:ecc2}
\nu|D^2 \bar u(q_0)|-C\bigl(1+|Du(x_0)|\bigr)\le |D^2 u(x_0)|\le \frac{1}{\nu}|D^2 \bar u(q_0)|+C\bigl(1+|Du(x_0)|\bigr)
\end{equation}
for some universal constants $\nu,C>0$. 
Since by assumption $Du$ is universally bounded inside $\Omega'$,
\eqref{eq:ecc} follows by \eqref{eq:ecc2} and \eqref{hessubar},
provided $M_1$ is sufficiently large.
\end{proof}

We  show now some geometric properties of sections
and some estimates for solutions of \eqref{eq:MA} which will play a major role
in the sequel. Here, the dilation of a section $S_h(x)$ is intended with respect to $x$.
\begin{proposition}[Properties of section]\label{secprop}
Let $u$ be a $c$-convex Aleksandrov solution of \eqref{eq:MA} with
$ 0 < \lambda \leq f \leq 1/\lambda$. Then, for any $\Omega' \subset \subset \Omega''\subset\subset \Omega$,
there exists a positive constant
$\rho=\rho(\Omega',\Omega'')$ such that
the following properties hold:
\begin{itemize}
\item[(i)] $S^u_h(x) \subset \Omega''$ for any $x\in \Omega'$, 
$0 \leq h\le 4\rho$.
\item[(ii)] There exist $0<\alpha_1< \alpha_2$ universal such that for all $\mu \in (0,1)$
\begin{equation*} 
 \mu^{\alpha_2} S^u_h(x) \subset S_{\lambda h}^u(x)\subset\mu^{\alpha_1} S_h^u(x)
\end{equation*} 
for any $x\in\Omega'$, $0 \leq2 h \le \rho$.
\item[(iii)] There exists a universal constant $\sigma<1$
such that, if $S_h^u(x)\cap S_h^u(y) \ne \emptyset$, then $S_h^u (y)\subset S_{h/\sigma}^u(x)$ 
for any $x,y\in\Omega'$, 
$0 \leq h \le \sigma\rho$.
\item[(iv)] $\cap_{0 <h \leq \rho} S_h^u(x)=\{x\}$.
\end{itemize} 
\end{proposition}

\begin{proof}
Points (i) and (iv) follow from the strict $c$-convexity of $u$ shown in \cite[section 7]{FKM},
and the fact that the modulus of strict $c$-convexity is universal (this last fact 
follows by a simple compactness argument in the spirit of \cite[Theorem 1']{CA2}).

Point (iii) corresponds the engulfing property of sections proved in \cite[Theorem 9.3]{FKM}. 

The second inclusion in point (ii) follows from \cite[Lemma 9.2]{FKM}\footnote{To be precise, in \cite{FKM} the dilation is done with respect to the center of the John ellipsoid,
and not with respect to the ``center'' $x$ of the section. However, it is easy to see that the same statement holds also in this case.}. For the first one,
it is enough to show that there exists a universal constant $\bar s \in (0,1)$ such that
\begin{equation}
\label{eq:bars}
\bar s Q^{\bar u}_h(\bar q) \subset Q^{\bar u} _{h/2}(\bar q)
\end{equation}
and then iterate this estimate (here $\bar u$ is defined  as in \eqref{baru},
and $\bar q:=q(x)$). To prove \eqref{eq:bars},
let $E_{2h}$ be the John ellipsoid associated to $Q^{\bar u}_{2h}(\bar q )$, and assume without loss of generality that  that $E_{2h}$ is centered at the origin. By convexity
of the sections in this new variables,
\[
\bar s (Q^{\bar u}_h(\bar q)-\bar q)+\bar q\subset Q^{\bar u}_h(\bar q)\subset  Q^{\bar u}_{2h}(\bar q)\subset nE_{2h} \qquad \forall\,\bar s \in (0,1).
\]
Observe now that, for any $q\in Q^{\bar u}_h(\bar q)$, we have (recall that $\bar u(\bar q)=0$)
\begin{equation}
\label{eq:barubars}
\bar u( \bar s (q-\bar q)+\bar q)=\bar s \int_0^1 D\bar u ((1-t\bar s)\bar q+t \bar s q)\cdot (q-\bar q)\,dt.
\end{equation}
Since $q,\bar q \in n E_{2h} $ we have $q-\bar q \in 2n E_{2h}$,  hence
 \begin{equation}\label{freddo}
 (q-\bar q)/2n \in E_{2h}\subset Q_{2h}(\bar q). 
 \end{equation}
 Moreover, by convexity of $Q_{2h}(\bar q)$,  $(1-t\bar s)\bar q+t \bar s q\in Q_h(\bar q)\subset \tau_0 Q_{2h}(\bar q) $ for some universal $\tau_0<1$ (see \cite[Lemma 9.2]{FKM}). Defining the ``dual norm'' $\|\cdot\|^*_{\mathcal K}$ associated to a convex set $\mathcal K$  as
\[
\| a\|^*_{\mathcal K} :=\sup_{\xi \in \mathcal K} a\cdot \xi,
\]
it follows from \cite[Lemma 6.3]{FKM} that
\begin{equation}\label{gradual}
\|D\bar u(q)\|^*_{Q_{2h}(\bar q)}
=\|-D_q\bar c(q, T_{\bar u}(q))\|^*_{Q^{\bar u}_{2h}(\bar q)}\le C h \quad \forall\, q\in Q^{\bar u}_{h}(\bar q).
\end{equation}
Thus, thanks to \eqref{eq:barubars} and \eqref{freddo} we get
\[
\begin{split}
\bar u( \bar s (q-\bar q)+\bar q)&=2n \bar s \int_0^1 D\bar u ((1-t\bar s)\bar q+t \bar s q)\cdot\frac{ (q-\bar q)}{2n}\,dt\\
&\le 2n \bar s \int_0^1\|D\bar u((1-t\bar s)\bar q+t \bar s q))\|^*_{Q_{2h}(\bar q)}dt\le  2n\bar s C h\le h/2,
\end{split}
\]
provided $\bar s$ is small enough. This proves the desired inclusion.
\end{proof}

As shown for instance in \cite{DFS}, an easy consequence of property (iii) is the following Vitali-type covering theorem. 

\begin{proposition}[Vitali covering theorem]\label{vitali}
Let $u,f,\Omega',\Omega'',\rho,\sigma$
be as in Proposition \eqref{secprop}, let $D$ be a compact subset of $\Omega'$,
and let $\{S_{h_x}(x)\}_{x \in D}$ be a family of sections with $h_x \leq \rho$.
Then we can find a finite number of these sections $\{S_{h_{x_i}}(x_i)\}_{i=1,\dots,m}$ such that $$D \subset \bigcup_{i=1}^m S_{ h_{x_i}}(x_i), \qquad \mbox{with $\{S_{\sigma h_{x_i}}(x_i) \}_{i=1,\dots,m}$ disjoint.}$$
\end{proposition}

We now want to show that sections at the same height have a comparable shape.
For this, we first recall the following estimate from \cite{FKM}:

\begin{proposition}[Aleksandrov estimates]
\label{prop:Alex}
Let $u,f,\Omega',\Omega'',\rho$ be be as in Proposition \eqref{secprop},
and let $S_h(x)$ be a section of $u$ for some $x \in \Omega'$ and $h\leq \rho$.
Then 
\begin{equation}\label{alest}
|S_h(x_0)|\approx h^{n/2}.
\end{equation}
\end{proposition}

\begin{remark}\label{rmk:gradientbound}
Estimates \eqref{gradual} and \eqref{alest} have the following important consequence: consider the function $\bar u$ defined in \eqref{baru}, fix one of its sections $Q_{h}$ such that $Q_{2h}\subset \Omega''$ with $\Omega''$ as above, normalize $Q_h$
using its corresponding John's transformation $L$, and define $\bar w$ as in \eqref{barbarw}.
Since $(\det L)^{-2/n}\approx |E_h|\approx \osc_{Q_{h}}\bar u\approx \osc_{Q_{2h}} \bar u $ (by \eqref{alest}) and $E_h\subset Q_{2h}$, we deduce the universal gradient bound
\begin{equation}\label{grad}
\begin{split}
\sup_{ L(Q_h)}|D \bar w |&=  (\det L)^{2/n}\sup_{Q_h}|(L^*)^{-1}D \bar u |\\
&\le C(\det L)^{2/n} \sup_{Q_h}\|D \bar u\|^*_{E_h}\\
&\le C(\det L)^{2/n} \sup_{Q_h}\|D \bar u\|^*_{Q_{2h}}\\
&\le C(\det L)^{2/n} \osc_{Q_{2h}} \bar u\le C.
\end{split}
\end{equation}
\end{remark}

\begin{lemma}\label{altezzecomp}
Let $u,f,\Omega',\Omega'',\rho$
be as in Proposition \eqref{secprop}. Then for any $0\leq h\leq \rho$
there exist two radii $r=r(h)$ and $R=R(h)$ such that, for every $x_0\in \Omega'$, if $E$ is the John ellipsoid associated to $S^u _h(x_0)$, then, up to  a translation,
\[
B_r(0)\subset E\subset B_R(0).
\]
\end{lemma}

\begin{proof} Let $r_1\le \ldots\le r_n$ be the axes of $E$. Since
$r_n\leq {\rm diam} (E)\le C$ and by \eqref{alest} 
$$
h^{n/2}
\approx |E|\approx r_1\cdot \ldots\cdot r_n \leq {\rm diam}(E)^{n-1}r_1,
$$
we obtain the desired lower bound on \(r_1\).
\end{proof}

Obviously analogous properties holds for the section $Q_h^{\bar u}(q_0)$. 
\begin{remark}
Notice that Proposition \ref{secprop}(ii) applied to the (convex) sections of $\bar u$ implies the following: given $x \in \Omega''$ and $h \leq \rho$,
let $r_1\leq \ldots\leq r_n$ denote the axes
of the John ellipsoid associated to $Q_h(x)$. Then
\begin{equation}
\label{eq:inclusions}
r_n\le Cr_1^{\alpha_3}
\end{equation}
for some universal exponent $\alpha_3<1$ and a constant $C(\Omega',\Omega'')$.

To see this just normalize $Q_\rho(x)$ using $L$ and notice that, by \cite[Theorem 6.11]{FKM} , ${\rm dist }\bigl(x,\partial \bigl( L(Q_\rho(x)\bigr)\bigr)\ge 1/C$ for some universal constant $C$. Thus, up to enlarge  $C$,
$$
\Bigg(\frac{h}{C\rho}\Bigg)^{\alpha_2}B_1(x)\subset L(Q_{h})\subset \Bigg(\frac{Ch}{\rho}\Bigg)^{\alpha_1}B_1(x).
$$
Since, by Lemma \ref{altezzecomp}, sections of height $\rho$ have bounded eccentricity (i.e., $|L|\approx C(\Omega',\Omega'')$), this implies the claim with $\alpha_3:=\alpha_2/\alpha_1$.

We now observe that  $\aalpha(Q_h)\approx r_n^2/(r_1\ldots r_n)^{2/n}$, from which we deduce 
 that 
$$
r_1^2\le C\frac{r_n^2} {\aalpha(Q_h)}.
$$
In particular, this and \eqref{eq:inclusions} imply 
$$
r_n \leq C r_1^{\a_3} \leq C\frac{r_n^{\a_3}} {\aalpha(Q_h)^{\a_3/2}},
$$
that is
$$
r_n \leq \frac{C} {\aalpha(S_h)^{\beta}},\qquad \text{with} \quad\beta:=\frac{\alpha_3}{2-2\alpha_3}.
$$
Hence, since $S_h$ is linked to $Q_h$ by a diffeomorphism with universal $C^1$ norm,
and ${\rm diam}(S_h)\le {\rm diam}(nE_{h})=2 n r_n$,
we get
\begin{equation}\label{diamsec}
{\rm diam}(S_h)\le \frac{\bar C} {\aalpha(S_h)^{\beta}}, \qquad \beta,\bar C>0 \text{ universal}.
\end{equation}
\end{remark}

\section{$W^{2,1+\kappa}$ estimates}
Applying first a large dilation to $\Omega$ we can assume that $B(0,1) \subset \Omega$,
and by a standard covering argument (see for instance \cite[Section 3]{DepFi})
it suffices to prove the $W^{2,1+\kappa}$ regularity of $u$ inside $B(0,1/2)$.
Also, by an
approximation argument\footnote{To approximate our solution with smooth ones, it suffices to regularize the data and then:\\
- either apply \cite[Remark 4.1]{LTWC2a}
(notice that, by Proposition \ref{secprop}(iv) and \cite[Theorem 8.2]{FKM}, $u$ is strictly $c$-convex and of class $C^1$ inside $\Omega$);\\
- or approximate our cost $c$ with cost functions satisfying
the strong version of {\bf (C1)} and apply \cite[Theorem 1.1]{LTWC2a}.},
 it is enough to prove the result when $u \in C^2$.
Hence Theorem \ref{w21eps} is a consequence of the following:

\begin{theorem}\label{W21epsloc}Let $u\in C^2$ be a $c$-convex solution of \eqref{eq:MA} with
$\Omega \supset B(0,1)$.
Then there exist universal constants $\kappa$ and $C$ such that
\begin{equation}\label{eqW21epsloc}
\int_{B(0,1/2)}|D^2 u|^{1+\kappa}\le C.
\end{equation} 
\end{theorem}

We start with the following lemma:
\begin{lemma}\label{lemmaL^1}Let $u$ be as above, $x_0\in B(0,3/4)$,
and $h>0$ such that $S_{2h}(x_0)\subset B(0,5/6)$. Consider the function $\bar u$ as in \eqref{baru}, its section $Q_h=Q_h(q_0)$ with $q_0:=q(x_0)$, and (up to a rotation) let
$E_h=\left\{\sum x_i^2/r_i^2\le 1\right\}$ be the John ellipsoid associated to $Q_h(q_0)$. Denote by $L$ be the affine transformation that normalizes $Q_h$, and define $\bar w$ and $\C_{ij,k\ell}$ as in \eqref{barbarw} and \eqref{CB} respectively. Then
\begin{equation}\label{L^1}
\int_{L(Q_h)}\big| \partial_{ij} \bar w-\C_{ij,k\ell} \partial_k \bar w \, \partial_\ell \bar w\big| \le C
\end{equation}
for some universal constant $C$.
\end{lemma}

\begin{proof}

Since, by the $c$-convexity of $u$ (which is preserved under change of variables),
the matrix $\bigl(\partial_{ij} \bar w
-\C_{ij,k\ell} \partial_k \bar w \, \partial_\ell \bar w\bigr)_{i,j=1,\ldots,n}$ is non-negative definite, it is enough to estimate
\[
\int_{L(Q_h)}\sum_{i=1}^n\big(\partial_{ii} \bar w-C_{ii,st} \partial_k\bar  w \, \partial_l\bar  w\big)
\]
from above.

Using the bounds $\mathcal H^{n-1}\bigl(\partial\bigl( L(Q_h)\bigr)\bigr)\le C(n)$
(since $L(Q_h)$ is a normalized convex set) and $|D\bar w|\le C$ (see \eqref{grad}),
we see that first term is controlled from above by
\begin{equation}\label{1}
\int_{L(Q_h)} \Delta \bar w = \int_{\partial\left( L(Q_h)\right)} D\bar w \cdot \nu \le \mathcal H^{n-1}\bigl(\partial\bigl( L(Q_h)\bigr)\bigr) \sup_{ L(Q_h)}|D \bar w |\le C.
\end{equation}
For the second term, we claim the following: there exists a universal constant $C$ such that
\begin{equation}\label{2}
\inf_{L(Q_h)}\sum_{i=1}^n  \C_{ii,k\ell} \partial_k \bar w \, \partial_\ell \bar w \ge -C
\end{equation}
To see this we write
$$
\sum_{i=1}^n  \C_{ii,k\ell} \partial_k \bar w \, \partial_\ell \bar w 
=\sum_{i=1}^n  \sum_{k,\ell\ne i} \C_{ii,k\ell} \partial_k \bar w \, \partial_\ell \bar w+2\sum_{i=1}^n\sum_{k\ne i}\C_{ii,ik} \partial_i \bar w \, \partial_k \bar w+\sum_{i=1}^n \C_{ii,ii} \partial_i \bar w \, \partial_i\bar w.
$$
We first observe that,
since for any $i=1,\ldots,n$ the vector $(\partial_1 \bar w,\dots,\partial_{i-1}\bar w,0,\partial_{i+1} \bar w,\dots,\partial_n\bar w) $ is orthogonal to coordinate vector $e_i$,
the first term in the right hand side is non-negative by condition {\bf (C1)}.

Concerning the second and the third term, taking into account the definition of $\C_{ij,k\ell}$ in \eqref{Cijst} we can rewrite them as
\begin{equation}
\label{eq:23term}
2\sum_{i=1}^n\sum_{k\ne i}(r_1 \dots r_n)^{2/n}\frac {r_i}{r_k} \B_{ii,ik} \partial_i \bar w \, \partial_k \bar w+\sum_{i=1}^n (r_1 \dots r_n)^{2/n}\B_{ii,ii} \partial_i \bar w \, \partial_i\bar w.
\end{equation}
Observe that, by \eqref{alest},
\begin{equation}
\label{eq:ri h}
(r_1 \dots r_n)^{2/n}\approx h.
\end{equation}
In addition, by the Lipschitz regularity of $\bar u$ (which is simply a consequence 
of the fact that
$u$ is locally Lipschitz inside $\Omega$),
\begin{equation}
\label{eq:ri h2}
h/r_k\leq C \qquad \forall\,k=1,\ldots,n.
\end{equation}
Since  $\|D\bar w\|_\infty\leq C$ (see \eqref{grad}) and the size of $\B$ is controlled by $\normm \bar c\normm \approx \normm  c\normm$,
by \eqref{eq:ri h} and \eqref{eq:ri h2}
we see that the expression in \eqref{eq:23term} is universally bounded.

This proves \eqref{2}, which combined with \eqref{1} concludes the proof.
\end{proof}

\begin{lemma}
\label{lem:hessian}With the same notation and  hypotheses  as in Lemma \ref{L^1}, let $\sigma$ be as in Proposition \ref{vitali}. Then there exists a universal constant $C$ such that 
\begin{equation}\label{mis}
\big|\{\tilde q\in L(Q_{\sigma h}):   {\rm Id}/C \le\partial_{ij} \bar w-\C_{ij,k\ell} \partial_k \bar w \, \partial_\ell \bar w  \le C\,{\rm  Id} \}\big|\ge \frac{1}{C}.
\end{equation}
\end{lemma}
\begin{proof}
Since $\sigma$ is universal and $L(Q_h)$ is normalized, by
Proposition \ref{secprop}(ii) we get
\[
|L(Q_{\sigma h})|\approx |L(Q_h)|\approx 1.
\]
So, using Lemma \ref{lemmaL^1} and Chebychev inequality,
we deduce the existence of a universal constant $C$ such that
\[
|\{\tilde q \in L(Q_{\sigma h}): \partial_{ij} \bar w-\C_{ij,k\ell} \partial_k \bar w \, \partial_\ell \bar w\le C\,{\rm Id} \}|\ge \frac{1}{C}.
\]
Since by \eqref{MAbarw} the product of the eigenvalues of the matrix
$\bigl(\partial_{ij} \bar w-\C_{ij,k\ell} \partial_k \bar w\bigr)_{i,j=1,\ldots,n}$ is of order one, whenever the eigenvalues
are universally bounded from above, they also have to be universally bounded also from below.
Hence, up to enlarging the value of $C$, this proves \eqref{mis}.
\end{proof}

\begin{remark}\label{undoscaling}
Recalling the definition \eqref{normalizedsize} of  $\aalpha( Q_h)=\aalpha(S_h)$, 
we can rewrite both \eqref{L^1} and \eqref{mis} in terms of
$\bar u$  and $Q_h=Q_h(q_0)$, obtaining that
\begin{equation*}
\int_{Q_h}\big| \partial_{ij} \bar u -\B_{ij,k\ell} \partial_k \bar u \, \partial_l \bar u\big| \le C\aalpha(Q_h)\Big|\big\{x\in Q_{\sigma h}:   \aalpha(Q_h)/C \le \big|\partial_{ij} \bar u-\B_{ij,k\ell} \partial_k \bar u \, \partial_l \bar u \big|\le C\aalpha(Q_h) \big\}\Big|
\end{equation*}
(see for instance the proof of \cite[Lemma 3.2]{DFS}).
In terms of $u$, this estimate becomes
\begin{equation}\label{main}
\int_{S_h}|D^2 u-\A(x,Du)| \le C_0\aalpha(S_h)\Big|\big\{x\in S_{\sigma h}:   \aalpha(S_h)/C_0 \le |D^2 u-\A(x,Du)|\le C_0\aalpha(S_h) \big\}\Big|,
\end{equation}
where $S_h=S_h(x_0)$ with $x_0$ an arbitrary point inside $B(0,3/4)$,
and $C_0$ is universal.
\end{remark}

\begin{proof}[Proof of Theorem \ref{W21epsloc}] Let  $M\gg1$ to be fixed later, set $R_{0}:=3/4$, and for all $m\ge 1$ define
\begin{equation}\label{rk}
R_m:=R_{m-1}-\bar CM^{-\beta}.
\end{equation}
with  $\bar C$ and $\beta$ as in \eqref{diamsec}. Let us use denote $B(0,R_m)$ by $B_{R_m}$,
set
\begin{equation}
\label{eq:F}
F(x):=\bigl|D^2 u(x)-\A(x,Du(x))\bigr|, 
\end{equation}
and define 
\begin{equation}\label{dk}
D_m:=\left\{x\in B_{R_m}:\ F(x)\ge M^m\right\}.
\end{equation}
Thanks to Proposition \ref{secprop}, there exists $\rho>0$ universal such that
$S_h(x)\subset B(0,5/6)$ for any $x\in B(0,{3/4})$ and $h \leq 2\rho$, and 
by Lemma \ref{altezzecomp} applied with $h=\rho$
we get $\aalpha(S_\rho(x))\approx 1$.
In addition,
since {\bf (C0)} implies that $|\A(x,Du)|\leq C_1$ inside $B(0,1)$
for some $C_1$ universal, we see that $\bigl||D^2u| - F\bigr| \leq C_1$.
Hence, using Lemma \ref{eccen}, we deduce that if $M \gg M_1+C_1$ then there exists
a small universal constant $\nu>0$ such that
$$
\aalpha(S_h(x))\ge \nu M^m \qquad \forall\, x \in D_m,\, \,h \leq \min\{\bar h(x),\rho\},\, \,m \geq 1.
$$
So, by choosing
 $M\ge \max\{ 1/\nu^4, M_1\}$ (so that $\nu M^{m+1} \geq M^{m+1/2}/\nu$),
by continuity we obtain that, for every point in $D_{m+1}$, there exists $h_x \in (0,\min\{\bar h(x),\rho\})$ such that $\aalpha(S_{h_x}) \in (\nu M^{m+1/2}, M^{m+1/2}/\nu)$.
In particular, by  \eqref{diamsec} we have
${\rm diam} (S_{h_x})\le \bar C M^{-\beta}$, which implies that (recall \eqref{rk})
\begin{equation}
\label{eq:inclusion}
\bigcup_{x\in D_{m+1}}  S_{h_x} \subset B\left(0,R_{m+1}+\bar C M^{-\beta}\right)= B_{R_m}. 
\end{equation}
According  to Proposition \ref{vitali}, we can cover $D_{m+1}$ with finitely many  sections $\{S_{h_{x_j}}\}_{x_j \in D_{m+1}}$ such that $ S_{\sigma h_{x_j}}$ are disjoint.  Then \eqref{eq:ecc} and \eqref{main} imply (recall \eqref{eq:F})
\begin{equation*}
\begin{split}
\int_{D_{m+1}}F  \le \sum_j \int_{ S_{h_{x_j}}}F&\le \sum_j C_0\,\aalpha(S_{h_{x_j}})|\{x\in  S_{\sigma h_{x_j}}:   \aalpha(S_{h_{x_j}})/C_0 \le F\le C_0\aalpha(S_{h_{x_j}}) \}|\\
&\le \sum_j \frac{C_0}{\nu} \,M^{m+1/2}|\{x\in  S_{\sigma h_{x_j}}:   \nu M^{m+1/2}/C_0 \le F\le C_0 M^{m+1/2}/\nu \}|.
\end{split}
\end{equation*}
Assuming now that $\sqrt{M}\ge  C_0/\nu$ and recalling \eqref{eq:inclusion}, we obtain
\begin{equation}\label{stimedk}
\begin{split}
\int_{D_{m+1}}F\le \sum_j  \int_{ S_{h_{x_j}}}F&\le \sum_j \frac{C_0}{\nu} \,M^{m+1/2}|\{x\in  S_{\sigma h_{x_j}}:   \nu M^{m+1/2}/C_0 \le F\le C_0 M^{m+1/2}/\nu \}|\\
&\le \sum_j \frac{C_0}{\nu} \,M^{m+1/2}|\{x\in  S_{\sigma h_{x_j}}:   M^m \le F\le M^{m+1} \}|\\
&= \sum_j \frac{C_0}{\nu} \,M^{m+1/2}|\{x\in  S_{\sigma h_{x_j}}:   M^m \le F\le M^{m+1} \}\cap B_{R_m}|\\
&\le  \frac{ C_0\sqrt{M}}{\nu} \int_{D_m\setminus D_{m+1}}F.
\end{split}
\end{equation}
Adding $\frac{C_0\sqrt{M}}{\nu}  \int_{D_{m+1}}F$ to both sides of the previous inequality,
we obtain
\[
\Bigg(1+\frac {C_0\sqrt{M}}{\nu}\Bigg) \int_{D_{m+1}}F \le \frac{ C_0\sqrt{M}}{\nu}  \int_{D_m}F. 
\]
which implies 
\[
\int_{D_{m+1}}F\le (1-\tau)\int_{D_m}F
\]
for some small constant $\tau=\tau(M)>0$.
We finally fix  $M$ so that it also satisfies 
\begin{equation}\label{k0}
\sum_{m\ge 1} \bar CM^{-m\beta}\le \frac 1 4.
\end{equation}
In this way $R_m \geq 1/2$ for all $m \geq 1$, so that the above inequalities and the definition of $R_m$ imply
\[
\int_{\{F\ge M^m\}\cap B(0,{1/2})}F\le \int_{D_m}F\le (1-\tau)^m\int_{D_{0}}F\le C(1-\tau)^m
\]
(here we used that $\int_{B(0,{3/4})}F\le C$,
which can be easily proved arguing as in the proof of Lemma \ref{L^1}). Thus,
choosing $\kappa>0$ such that $1-\tau=M^{-2\kappa}$,
we deduce that 
\[
|\{F\ge t\}\cap B(0,{1/2})|\le\frac{1}{t} \int_{\{F\ge t\}\cap B(0,{1/2})}F \le Ct^{-1-2\kappa}
\]
for some $C>0$ universal, 
 which implies that $F \in L^{1+\kappa}(B(0,{1/2}))$.
 Recalling the definition of $F$ (see \eqref{eq:F}) and that $|\A(x,Du)|\leq C$ inside $B(0,{1})$ (by {\bf (C0)}),
 this concludes the proof.
\end{proof}

\end{document}